\documentclass[11pt,psamsfonts]{article}
\usepackage{a4,fullpage,amssymb,epsf,psfrag,times}
\usepackage{graphicx}
\usepackage{graphics}
\usepackage{amsmath}
\usepackage{amsfonts}
\usepackage{enumerate}
\usepackage{latexsym}
\usepackage{epsfig}
\usepackage{amsthm}
\usepackage{amsmath}
\usepackage{amssymb}
\usepackage{latexsym}
\usepackage{epsfig}
\usepackage{amssymb,latexsym}
\usepackage[all]{xy}
\usepackage{amscd}
\usepackage{amsmath}
\usepackage{amsfonts}
\usepackage{enumerate}
\usepackage{latexsym}
\usepackage{epsfig}
\usepackage{amsthm}
\usepackage{amsmath}
\usepackage{amssymb}
\usepackage{latexsym}
\usepackage{epsfig}
\usepackage[T1]{fontenc} 

\makeatletter\@addtoreset{figure}{section}\makeatother
\makeatletter\@addtoreset{table}{section}\makeatother

\newtheorem{theorem}{Theorem}[]

\newtheorem{cor}[theorem]{Corollary}

\newenvironment{remark}{\refstepcounter{theorem}\par\medskip\noindent{\bf
Remark~\thetheorem~~}}{\unskip\nobreak\hfill\hbox{ $\oslash$}\par\bigskip}

\newenvironment{example}{\refstepcounter{theorem}\par \medskip \noindent{\bf
Example~\thetheorem~~}}{\unskip\nobreak\hfill\hbox{ $\oslash$}\par\bigskip}

\newcommand{\op}[1]{\!\!\mathop{\rm ~#1}\nolimits}

\begin{document}

\title{Applying Hodge theory to detect Hamiltonian flows}

\author{}
\author{\'Alvaro Pelayo${}^1$ and Tudor S. Ratiu${}^2$}
\addtocounter{footnote}{1} 

\footnotetext{Partially supported by an NSF postdoctoral 
fellowship and the program \textit{Symplectic and Contact
Geometry and Topology} at MSRI.
\addtocounter{footnote}{1} }

\footnotetext{ 
 Partially
supported by a Swiss NSF grant 
and the program \textit{Symplectic and Contact
Geometry and Topology} at MSRI.
\addtocounter{footnote}{1} }

\date{}

\maketitle

\begin{abstract}
We prove that when Hodge theory survives on non-compact
symplectic manifolds, a compact symplectic Lie group action
having fixed points is necessarily Hamiltonian, provided 
the associated almost complex structure preserves the space
of harmonic one-forms. For example, this is the case for
complete K\"ahler manifolds for which the symplectic
form has an appropriate decay at infinity. This extends a classical theorem of Frankel for
compact K\"ahler manifolds  to complete 
non\--compact K\"ahler manifolds. 
\end{abstract}

\section{Introduction}

The study of fixed points of dynamical systems and
group actions is a classical topic studied in
geometry. A seminal result of T. Frankel \cite{Frankel1959}
states that if a symplectic circle action on a compact
connected K\"ahler manifold has fixed points, then
it must be Hamiltonian. The present paper builds on
Frankel's ideas to study whether this striking result
persists under some reasonable conditions for possibly 
non-K\"ahler, non-compact symplectic manifolds. The 
non\--compact case is of special
interest in dynamical systems.

Loosely speaking, the goal of this paper is 
to prove that \emph{when Hodge theory survives on 
non-compact symplectic manifolds and the space of harmonic
one-forms is preserved by an associated almost complex structure}, 
the existence of a fixed point for a symplectic
action of a compact Lie group forces the action to
be Hamiltonian. This has particularly strong implications for
complete K\"ahler manifolds.
 
 All manifolds in this note are 
paracompact and boundaryless. 

\subsection{Main Theorem}

Let $(M,\omega)$ be a symplectic manifold.
The triple $(\omega,g,\mathbf{J})$ is a \emph{compatible triple} on $(M,\omega)$ if $g$ is a Riemannian
metric and $\mathbf{J}$ is an almost
complex structure (i.e., a vector bundle automorphism
$\mathbf{J} \colon \op{T}M \to \op{T}M$) satisfying
$\mathbf{J}^2=-\op{Identity}$)  such that
$g(\cdot,\cdot)=\omega(\cdot,\mathbf{J}\cdot)$.
For the following theorem recall that the standard construction
of a compatible triple from a symplectic form immediately
extends to the $G$-invariant case.

Let $G$ be a Lie group with Lie algebra $\mathfrak{g}$.
Suppose that $G$ acts on $M$ by symplectomorphisms (i.e., diffeomorphisms which preserve the symplectic form).
Any element
$\xi \in \mathfrak{g}$ generates a vector field $\xi_M$
on $M$, called the \emph{infinitesimal generator}, given
by $\xi_M(x):= \left.\frac{d}{dt}\right|_{t=0}  
\op{exp}(t\xi)\cdot x$, where $\op{exp} \colon
\mathfrak{g} \to G$ is the exponential map and
$x \in M$.
The  $G$-action on 
$(M,\omega)$  is said to 
be \emph{Hamiltonian} if there exists a smooth
equivariant map $\mu 
\colon M \to \mathfrak{g}^*$, called the \emph{momentum map}, 
such that for
all $\xi \in \mathfrak{g}$ we have  $\mathbf{i}_{\xi_M}
\omega : = \omega(\xi_M, \cdot) = \mathbf{d} \langle \mu, 
\xi \rangle$, where $\left\langle \cdot , \cdot 
\right\rangle : \mathfrak{g}^\ast \times \mathfrak{g} 
\rightarrow \mathbb{R}$ is the duality pairing. For
example, if $G\simeq (S^1)^k$, $k \in \mathbb{N}$, is a torus, the existence
of such a map $\mu$ is equivalent to the exactness of
the one-forms $\mathbf{i}_{\xi_M}\omega$ for all $\xi \in
\mathfrak{g}$. In this case the obstruction of the action 
to being Hamiltonian lies in the first de Rham cohomology
group of $M$. The simplest example of a $S^1$-Hamiltonian action is rotation of the sphere $S^2$ about the polar axis;
see Figure \ref{fig1}. The flow lines of the infinitesimal
generator defining this action are the latitude circles.

\begin{figure}[htbp]
  \begin{center}
    \includegraphics[width=7cm]{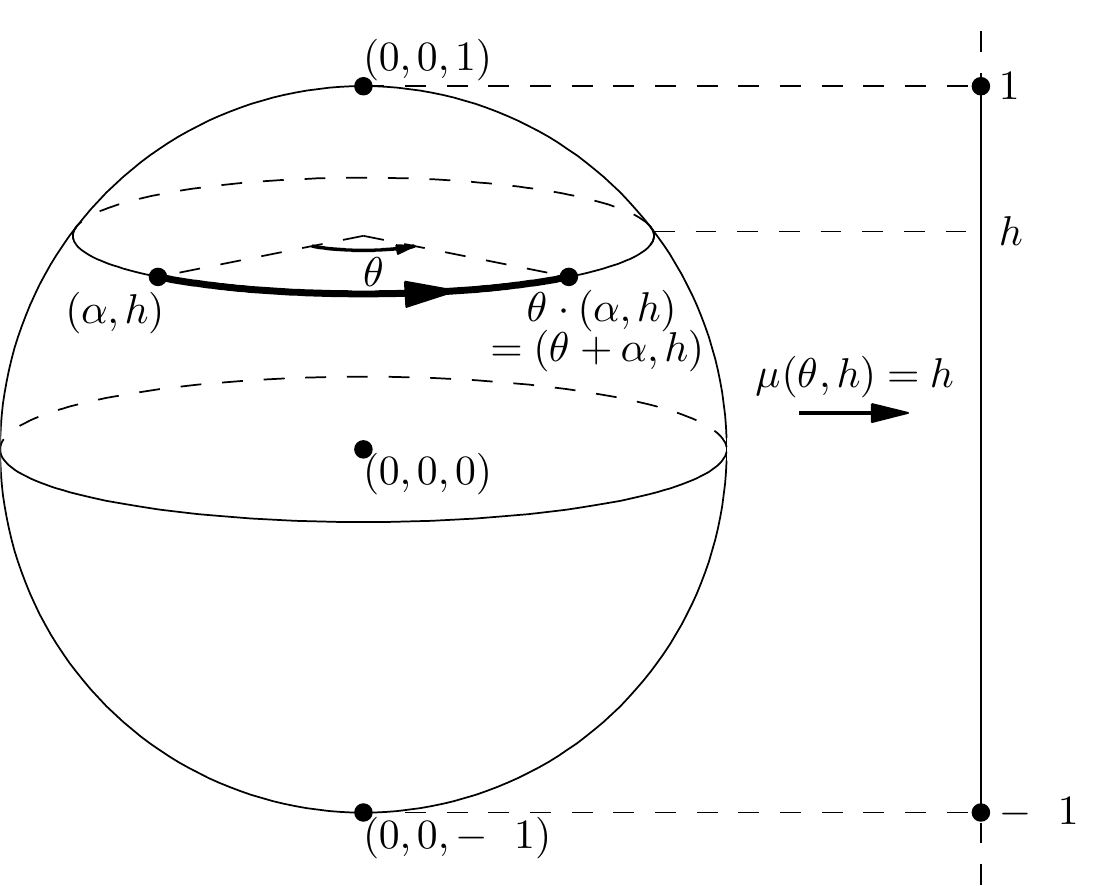}
   \caption{Momentum map for the $S^1$-action on $S^2$}
    \label{fig1}
  \end{center}
\end{figure}

We denote by $\operatorname{L}^2_\rho$ the $\operatorname{L}^2$-Hilbert space of square integrable functions relative 
the the measure $\rho$.

\begin{theorem}\label{thm_gen}
Let $G$ be a compact Lie group acting on a symplectic manifold $(M,\omega)$ by means of symplectomorphisms.
Let $(\omega, g, \mathbf{J})$ be a 
$G$-invariant compatible triple.
Let $\lambda$ be a measure on $M$ such that the 
Radon-Nikodym derivative of $\lambda$ relative to the Riemannian measure is a bounded
function on $M $ and denote by $\delta_\lambda $ the formal
adjoint of $\mathbf{d}$ relative to the 
$\operatorname{L}^2_ \lambda$ inner product. 

Assume that
each $\operatorname{L}^2_ \lambda$ closed one-form 
decomposes $\operatorname{L}^2_ \lambda$-orthogonally as a sum of
the differential of a $\operatorname{L}^2_ \lambda$ smooth function
and a harmonic $\operatorname{L}^2_ \lambda$ one-form (i.e., in the joint kernel of $\mathbf{d}$ and $\delta_\lambda$) and that each cohomology class of a closed one-form in $\operatorname{L}^2_ \lambda$ 
has a unique harmonic  representative.
If $\mathbf{J}$ preserves harmonic one-forms and the $G $-action has fixed points on every connected component, then the action is Hamiltonian. 
\end{theorem}

As far as the authors know
this is the first instance in which the relation between
the existence of fixed points and the Hamiltonian character
of the $G$\--action has been studied for non-compact manifolds. All assumptions 
of the theorem, with possibly
the exception of the existence of fixed points, hold
for compact K\"ahler manifolds.

\begin{remark}
Note that the theorem is implied by the case $G = S^1$
using a standard argument based on the fact that every 
point of a compact Lie group lies on a maximal torus. 
We recall the proof. First, note that if the theorem 
holds for $S^1$, then it also holds for any torus 
$\mathbb{T}^k: = (S^1)^k$, $k \in \mathbb{N}$,
since the momentum map of the product of two Hamiltonian 
actions is the sum of the two momentum maps. Now let $G$
be an arbitrary compact Lie group whose symplectic
action on $M$ has at least a fixed point. If $\xi \in 
\mathfrak{g}$, then $\exp \xi$ necessarily lies in a maximal 
torus and the restriction of the action to the torus has
fixed points. Since the conclusion of the theorem holds
for symplectic torus actions, it follows that this restricted
action has an invariant momentum map. In particular, 
$\mathbf{d}\mathbf{i}_{\xi_M} \omega = \mathbf{d}f^\xi$ 
for some $f ^\xi \in \operatorname{C} ^{\infty}(M)$, 
a relation valid for every $\xi\in \mathfrak{g}$. Using 
a basis $\{e_1, \ldots, e_r\}$ of 
$\mathfrak{g}$, we define a new map $\mu:M 
\rightarrow \mathfrak{g}^\ast$ by $\mu^ \xi := 
\xi^1f^{e_1}+ \cdots + \xi^rf^{e_r}$, where 
$\xi= \xi^1e_1+ \cdots + \xi^re_r$.
We clearly have $\mathbf{i}_{\xi_M} \omega = \mathbf{d}\mu^ 
\xi$ which proves that $\mu: M \rightarrow 
\mathfrak{g}^\ast$, defined by the requirement that 
its $\xi$-component is $\mu^\xi$ for each 
$\xi\in \mathfrak{g}$, is a momentum map of the $G $-action. 
Since $G$ is compact, one can construct out of $\mu$ an  
\emph{equivariant} momentum map (see, e.g., 
\cite[Theorem 11.5.2]{MaRa2003}), which proves that the 
action is Hamiltonian.
\end{remark}

\begin{example}
 The assumption 
that the action has fixed points in Theorem \ref{thm_gen}
is essential. 
For example, the $S^1$-action
on $\mathbb{T}^2$ given by $e^{2 i \varphi} \cdot 
(e^{2 i\theta_1}, e^{2 i\theta_2}): = (e^{2 i\theta_1}, 
e^{2 i(\theta_2+ \varphi)})$ is a 
symplectic action on a K\"ahler manifold which is free and 
hence has no fixed points.  Recall that the range of the derivative of the momentum
map at a given point equals the annihilator of the symmetry
algebra of that point (the Reduction Lemma). If $G=S^1$, since the manifold is compact, the momentum map
must have critical points which
shows that the action in this example does
not admit a momentum map. 
\end{example}

\subsection{Consequences in the K\"ahler case}

If the manifold is K\"ahler, the associated complex structure
automatically preserves the space of harmonic one-forms.
Since for compact manifolds one always has the Hodge decomposition for the measure $\omega^n$, $2n = \dim M$, we 
immediately conclude the following statement.

\begin{cor} 
Let $(M, \omega)$ be a compact symplectic $G $-space, where
$G $ is a compact Lie group. Assume that the the space of harmonic one-forms is invariant under the complex structure
(which always holds if $M $ is a K\"ahler manifold).
If the $G $-action has fixed
points on every connected component of $M $ then it is
Hamiltonian.
\end{cor}

A consequence of the proof of Theorem \ref{thm_gen} is the following result whose proof is given at the end of Section \ref{sec_proof}.

\begin{cor} \label{nicecor}
Let $(M, \omega)$ be a $2n$-dimensional complete connected 
K\"ahler $G $-space, where $G $ is a compact Lie group. 
If the contraction of $\omega$ with all infinitesimal 
generators of the action is in $\op{L}^2_{\omega^n}$ and the $G$-action has fixed points then it is Hamiltonian.
\end{cor}


\subsection{Frankel's Theorem and further results}

The first result concerning the relationship between the existence of fixed points and the Hamiltonian character of the action is
Frankel's celebrated theorem \cite{Frankel1959} stating that if the manifold
is compact, connected, and K\"ahler, $G=S^1$, and
the symplectic action has fixed points, then it must be Hamiltonian (note that $\mathbf{J}
\mathcal{H} \subset \mathcal{H}$ holds, see \cite[Cor 4.11, Ch. 5]{Wells2008}). Frankel's influential work has inspired subsequent research.
McDuff \cite[Proposition 2]{McDuff1988} has shown that any symplectic
circle action on a compact connected symplectic 4-manifold
having fixed points
is Hamiltonian. However, the result is false in higher
dimensions since she gave an example 
(\cite[Proposition 1]{McDuff1988}) of a compact connected symplectic 
6-manifold with a symplectic circle action which has 
fixed points (formed by tori), but is not Hamiltonian. If the 
$S^1$-action is semifree (that is, it is free off the fixed
point set), then Tolman and Weitsman 
\cite[Theorem 1]{ToWe2000} have shown that any symplectic $S^1$-action on 
a compact connected symplectic manifold having fixed points 
is Hamiltonian. Feldman \cite[Theorem 1]{Feldman2001}
characterized the obstruction for
a symplectic circle action on a compact manifold to be
Hamiltonian and deduced the McDuff and Tolman-Weitsman theorems by applying his
criterion. He showed that the Todd genus of a manifold admitting a symplectic circle action with isolated
fixed points is equal either to 0, in which case the action is non-Hamiltonian, or to 1, in which
case the action is Hamiltonian. In addition, any symplectic circle action on a manifold with positive Todd genus is Hamiltonian. For additional results regarding aspherical
symplectic manifolds (i.e.  $\int_{S^2} f ^\ast \omega = 0$ 
for any smooth map $f: S^2 \rightarrow M$) see \cite[Section 8]{KeRuTr2008} and \cite{LuOp1995}.
As of today, there are no known examples of symplectic $S^1$-actions on compact connected symplectic manifolds that are not Hamiltonian but have fixed points.

For higher dimensional Lie groups, less is known.
Giacobbe \cite[Theorem 3.13]{Giacobbe2005} proved that a symplectic
action of a $n$-torus on a 
$2n$-dimensional compact connected symplectic manifold with fixed points is necessarily Hamiltonian;
see also \cite[Corollary 3.9]{DuPe2007}. If $n=2 $ this
result can be checked explicitly from the classification 
of symplectic 4-manifolds with symplectic 2-torus actions
given in \cite[Theorem 8.2.1]{Pelayo2010} (since cases 2--5 in the statement of the theorem are
shown not to be Hamiltonian; the only non-K\"ahler
cases are given in items 3 and 4 as proved in 
\cite[Theorem 1.1]{DuPe2010}). 

If $G$
is a Lie group with Lie algebra $\mathfrak{g}$ acting 
symplectically on the symplectic manifold $(M, \omega) $, 
the action is said to be \textit{cohomologically free} if 
the Lie algebra homomorphism
$\xi \in \mathfrak{g} \mapsto [\mathbf{i}_{\xi_M} \omega] \in  \operatorname{H}^1(M, \mathbb{R})$ is 
injective; $\operatorname{H}^1(M, \mathbb{R})$ is regarded as an abelian
Lie algebra.  Ginzburg \cite[Proposition 4.2]{Ginzburg1992}
showed that if a torus $\mathbb{T}^k = (S^1)^k$, $k \in \mathbb{N}$, acts symplectically,
then there exist subtori $\mathbb{T}^{k-r}$, $\mathbb{T}^r$
such that $\mathbb{T}^k =\mathbb{T}^r\times\mathbb{T}^{k-r}$,  the $\mathbb{T}^r$-action is cohomologically free,
and the $\mathbb{T}^{k-r}$-action is Hamiltonian. This homomorphism is the
obstruction to the existence of a momentum map: it vanishes
if and only if the action admits a momentum map. For compact
Lie groups the previous result holds only up to coverings.
If $G$ is a compact Lie group, then it is well-known that
there is a finite covering $\mathbb{T}^k \times K \rightarrow 
G $, where $K$ is a semisimple compact Lie group. So
there is a symplectic action of $\mathbb{T}^k \times K$
on $(M, \omega)$. The $K$-action is Hamiltonian, since
$K$ is semisimple. The previous result applied to 
$\mathbb{T}^k$ implies that there is a finite covering 
$\mathbb{T}^r\times (\mathbb{T}^{k-r} \times K) \rightarrow 
G$ such that the $(\mathbb{T}^{k-r} \times K)$-action is
Hamiltonian and the $\mathbb{T}^r$-action is cohomologically
free; this is \cite[Theorem 4.1]{Ginzburg1992}. The Lie
algebra of $\mathbb{T}^{k-r} \times K$ is $\ker\left(\xi
\mapsto [\mathbf{i}_{\xi_M} \omega]\right)$.

\textbf{Acknowledgements.} We thank Ian Agol, Denis Auroux, 
Dan Halpern\--Leistner, and Alan Weinstein for helpful
discussions.

\section{Proofs of Theorem \ref{thm_gen} and Corollary \ref{nicecor}}
\label{sec_proof}

The proof extends Frankel's method \cite{Frankel1959} to the case of non-compact manifolds.

Let $G$ be a compact Lie group acting on the symplectic manifold $(M,\omega)$ by means of symplectic diffeomorphisms.  Let $(\omega, g, \mathbf{J})$ be a 
$G$-invariant compatible triple. 
Let $\lambda$ be a measure on $M $ such that the 
Radon-Nikodym derivative of the Riemannian measure with respect to $\lambda$ is a bounded
function on $M $ and denote by $\delta_\lambda $ the formal
adjoint of $\mathbf{d}$ relative to the 
$\operatorname{L}^2_ \lambda$ inner product, that is,
\[
\int_M \left\langle \! \left\langle \mathbf{d} \alpha, \beta \right\rangle \! \right\rangle
d \lambda =
\int_M \left\langle \! \left\langle \alpha, \delta_\lambda \beta \right\rangle \! \right\rangle d \lambda,
\]
for all $\alpha \in \Omega^q(M)$, $\beta \in \Omega^{q+1}(M) $,
where $\left\langle \! \left\langle\,, \right\rangle \! \right\rangle$ is the inner product
on forms.
Let $\|\cdot \|_{\operatorname{L}^2_\lambda}$ be the $\operatorname{L}^2$-norm on forms
relative to the measure $\lambda$. 

By assumption we have:
\begin{itemize}
\item[{\rm (i)}] 
Any $\alpha\in \Omega^1(M)$ such that $\|\alpha\|_{\operatorname{L}^2_\lambda} < \infty$ and $ \mathbf{d}\alpha = 0$ has a unique
$\operatorname{L}^2_ \lambda$-orthogonal decomposition $\alpha = 
\mathbf{d}f + \chi$, where $f \in C ^{\infty}(M)$, $\mathbf{d}f \in 
\operatorname{L}^2_{\lambda}(M) $, $\mathbf{d}\chi=0$, $\delta_\lambda \chi
=0 $, $\chi\in \operatorname{L}^2(\wedge^1M, g)\cap \Omega^1(M)$. Such forms $\chi$ are
called \emph{harmonic}. Let $\mathcal{H}$ denote the space of
harmonic one-forms.
\item[{\rm (ii)}] If a cohomology class $[\alpha] \in \operatorname{H}^1(M, \mathbb{R})$ with $\| \alpha\|_{ \operatorname{L}^2_\lambda} < \infty$ has a
harmonic representative, it is necessarily unique.
\item[{\rm (iii)}] $\mathbf{J}\mathcal{H} \subset \mathcal{H}$.
\end{itemize}

\begin{remark}
Condition (ii) can be reformulated as:
\begin{itemize}
\item[{\rm (ii')}] If $f \in \operatorname{C}^{\infty}(M)$,
$\| \mathbf{d} f\|_{\operatorname{L}^2_\lambda}<\infty$, and 
$\delta_\lambda \mathbf{d} f = 0 $ then $f$ is
a constant function on each connected component of $M $.
\end{itemize}

Indeed, suppose (ii') holds and let $\alpha $ and $\beta$ 
be two harmonic
representatives of the same cohomology class with finite $\operatorname{L}^2_\lambda$-norm, then 
$\alpha - \beta = \mathbf{d}f$ for some $f \in 
\operatorname{C} ^{\infty}(M)$, 
$\|\mathbf{d}f\|_{\operatorname{L}^2_\lambda} < \infty$. Therefore
$\delta_\lambda \mathbf{d}f = \delta_\lambda( \alpha - \beta)
= 0$ and hence, by (ii'), it follows that $f$ is constant
on each connected component of $M $ implying that $\alpha = \beta$. Conversely, if $\|\mathbf{d}f\|_{\operatorname{L}^2_\lambda} < \infty$ and
 $\delta_\lambda \mathbf{d} f = 0 $, then $\mathbf{d} f$
is a smooth $\operatorname{L}^2_\lambda$ harmonic one-form 
representing the zero cohomology class. Thus, by (ii),
$f$ is constant on each connected component of $M$.
\end{remark}

\medskip

We want to show that if the $G $-action has fixed points on every connected
component of $M $, then the action is Hamiltonian.
\medskip

\noindent\textbf{Proof of Theorem \ref{thm_gen}.} We divide 
the proof in four steps.
\smallskip

\noindent \textbf{Step 1} (Vanishing of harmonic one\--forms along infinitesimal generators). We show that if $\alpha \in 
\Omega^1(M)$ is harmonic and $\| \alpha\|_{\operatorname{L}^2_\lambda} < \infty$, then $\boldsymbol{\pounds}_{ \xi_M}\alpha=0$
which is a 
standard result for the usual codifferential (Killing
vector fields preserve harmonic one-forms). Since, in our
case, we use $\delta_\lambda$ instead of the usual codifferential we give the proof. 

We begin with the following
observation: if $\varphi:M \rightarrow M$ is an isometry and
preserves the measure $\operatorname{d}\!\lambda$, that
is, $\varphi^\ast g = g$ and $\varphi^\ast(\operatorname{d}\! \lambda) =  \operatorname{d}\!\lambda$, then
\begin{equation}
\label{useful_formula}
\varphi^\ast \left(\left\langle \! \left\langle \nu, \rho \right\rangle \! \right\rangle \operatorname{d}\!\lambda\right) = 
\left\langle \! \left\langle \varphi ^\ast \nu, \phi^\ast \rho \right\rangle \! \right\rangle
\operatorname{d}\!\lambda
\end{equation}
for any $\nu, \rho \in \Omega^1(M)$. 

Next, let
$F_t := \Phi_{\exp(t \xi)} $ be the flow of $\xi_M$ 
which is an isometry of $(M, g)$.
Since $\mathbf{d} \alpha = 0 $ it follows that $\mathbf{d}
F  _t^\ast \alpha = F _t^\ast \mathbf{d}\alpha = 0$.
We now show that $F_t$ commutes with $\delta_\lambda$. Indeed,
for any $\beta, \gamma \in \Omega^1(M) $, we have
\begin{align*}
\left\langle \delta_\lambda F_t^\ast \beta, \gamma 
\right\rangle_{\operatorname{L}^2_{\lambda}} 
& = \int_M \left\langle \! \left\langle F_t^\ast\beta, \mathbf{d} \gamma
\right\rangle \! \right\rangle \operatorname{d}\!\lambda 
\stackrel{\eqref{useful_formula}}=
\int_M F _t^\ast \left(\left\langle \! \left\langle \beta, (F_t)_* \mathbf{d}
\gamma \right\rangle \! \right\rangle \operatorname{d}\!\lambda\right)
= \int_M \left\langle \! \left\langle \beta, (F_t)_* \mathbf{d}
\gamma \right\rangle \! \right\rangle \operatorname{d}\!\lambda \\
& 
= \int_M \left\langle \! \left\langle \delta_\lambda\beta, (F_t)_* 
\gamma \right\rangle \! \right\rangle \operatorname{d}\!\lambda
\stackrel{\eqref{useful_formula}}=
\int_M (F_t)_* \left(\left\langle \! \left\langle F _t^\ast \delta_\lambda
\beta, \gamma \right\rangle \! \right\rangle \operatorname{d}\! \lambda\right)
= \left\langle F_t^\ast \delta_\lambda  \beta, \gamma 
\right\rangle_{\operatorname{L}^2_{\lambda}}
\end{align*}
and hence $\delta_\lambda F_t^\ast \beta = 
F_t^\ast \delta_\lambda \beta$. In particular, 
$\delta_\lambda\alpha = 0$ implies that $\delta_\lambda
F _t^\ast \alpha = F _t^\ast \delta_\lambda \alpha = 0$,
which shows that $F _t^\ast \alpha$ is harmonic.

However, in $\op{H}^1(M, \mathbb{R}) $ we have
$[F_t^\ast \alpha] = F_t^\star[ \alpha] = [\alpha]$ 
since $F _t$ is isotopic
to the identity; here $F_t^\star $  denotes the isomorphism induced by the diffeomorphism $F _t$ on the cohomology groups. Therefore, the relation 
$[F_t^\ast \alpha] = [\alpha]$ implies that $F _t^\ast \alpha = \alpha$ since
both $F _t^\ast\alpha$ and $\alpha$ are harmonic and each cohomology
class has a unique harmonic representative by hypothesis (ii). Taking the $t $-derivative
implies that $\boldsymbol{\pounds}_{ \xi_M} \alpha =0$, as required.
\smallskip

\noindent\textbf{Step 2} (Using the existence of fixed points). Define $\xi_M^\flat: = g( \xi_M, \cdot ) \in \Omega^1(M) $. If $\alpha\in \Omega^1(M)$ is harmonic and $\| \alpha\|_{ \operatorname{L}^2_{ \lambda}} < \infty$,
it follows from Step 1 that $0 = \boldsymbol{\pounds}_{\xi_M} \alpha = \mathbf{i}_{\xi_M} \mathbf{d} \alpha + \mathbf{d}\mathbf{i}_{\xi_M} \alpha = \mathbf{d}\mathbf{i}_{\xi_M} \alpha$. Thus $\alpha(\xi_M)$ is constant on each connected component of
$M$. At this point we use the crucial hypothesis that the
group action has at least one fixed point on each connected
component. Thus, $\alpha(\xi_M) =0$ on $M$. Therefore, 
\begin{equation}
\label{vanishing}
\left\langle \xi_M^\flat, \alpha \right\rangle_{\op{L}^2_\lambda}
= \int_M \alpha(\xi_M) \operatorname{d}\! \lambda = 0
\end{equation}
for any harmonic one-form $\alpha$ satisfying $\| \alpha\|_{\operatorname{L}^2_{ \lambda}} < \infty$, where $\dim M = 2n$. 
\smallskip

\noindent\textbf{Step 3} (Applying the existence of a Hodge decomposition). Since $\mathbf{d}\mathbf{i}_{\xi_M} \omega=0$ and $\|\mathbf{i}_{\xi_M} \omega\|_{\operatorname{L}^2_{ \lambda}} < \infty$, by
hypothesis (i) we have  $\mathbf{i}_{\xi_M} \omega = 
\mathbf{d} f^ \xi + \chi^ \xi$, where 
$f^ \xi \in \op{C}^{\infty}(M) $,   
$\chi^ \xi \in \Omega^1(M)$ is harmonic, $\|\mathbf{d} f ^\xi\|_{\operatorname{L}^2_{ \lambda}} < \infty$, and 
$\|\chi^\xi\|_{\operatorname{L}^2_{ \lambda}} < \infty$. 

Let us prove that 
$\chi^\xi=0$. To do this, recall
that $\mathbf{J}$ is defined on one-forms by the relation 
$(\mathbf{J} \beta)(X) = \beta(\mathbf{J}X) $ for $\beta
\in \Omega^1(M)$ and $X \in \mathfrak{X}(M)$. Therefore,
$\mathbf{i}_{\xi_M} \omega = -\mathbf{J}\xi_M^\flat$. Indeed,
for any $Y \in \mathfrak{X}(M) $ we have
\begin{align*}
(\mathbf{i}_{\xi_M} \omega)(Y) = \omega(\xi_M, Y)
= - \omega(\xi_M, \mathbf{J}( \mathbf{J}Y))
= - g(\xi_M, \mathbf{J}Y) 
= - \xi_M^\flat(\mathbf{J}Y)
= - (\mathbf{J}\xi_M^\flat)(Y).
\end{align*}
Let $ \alpha
\in \Omega^1(M) $ be an arbitrary harmonic one-form such 
$\|\alpha\|_{\operatorname{L}^2_{ \lambda}} < \infty$. Then
\begin{align*}
\left\langle \mathbf{i}_{\xi_M} \omega, \alpha 
\right\rangle_{\operatorname{L}^2_ \lambda}
= \left\langle - \mathbf{J}\xi^ \flat_M, \alpha 
\right\rangle_{\operatorname{L}^2_ \lambda}
= -\int_M \left\langle \! \left\langle \mathbf{J}\xi_M^\flat, \alpha 
\right\rangle \! \right\rangle \operatorname{d}\! \lambda
= -\int_M \left\langle \! \left\langle \xi_M^\flat, \mathbf{J}\alpha 
\right\rangle \! \right\rangle \operatorname{d}\! \lambda
= -\left\langle \xi_M^\flat, \mathbf{J}\alpha 
\right\rangle_{\op{L}^2_\lambda}.
\end{align*}
By hypothesis (iii), $\mathbf{J} \alpha$ is harmonic and
therefore $\left\langle \xi_M^\flat, \mathbf{J}\alpha 
\right\rangle_{\op{L}^2_\lambda} = 0 $ by \eqref{vanishing}.
Using again hypothesis (i), we conclude that $\chi^\xi =0$.
Therefore
$\mathbf{i}_{\xi_M} \omega = \mathbf{d}f^ \xi $ for
any $\xi\in \mathfrak{g}$ and both sides of this identity
are linear in $\xi \in \mathfrak{g}$.
\smallskip

\noindent\textbf{Step 4} (Construction of an equivariant momentum map). Using a basis 
$\{e_1, \ldots, e_r\}$ of 
$\mathfrak{g}$, we define a new map $\mu:M 
\rightarrow \mathfrak{g}^\ast$ by $\mu^ \xi := 
\xi^1f^{e_1}+ \cdots + \xi^rf^{e_r}$, where 
$\xi= \xi^1e_1+ \cdots + \xi^re_r$.
We clearly have $\mathbf{i}_{\xi_M} \omega = \mathbf{d}\mu^ 
\xi$ which proves that $\mu: M \rightarrow 
\mathfrak{g}^\ast$, defined by the requirement that 
its $\xi$-component is $\mu^\xi$ for each 
$\xi\in \mathfrak{g}$, is a momentum map of the $G $-action. 

Since $G$ is compact, one can construct out of $\mu$ an  
equivariant momentum map (see, e.g., 
\cite[Theorem 11.5.2]{MaRa2003}, which shows that the 
action is Hamiltonian thereby completing the proof of the theorem.
\medskip
\medskip

\noindent \textbf{Proof of Corollary \ref{nicecor}.} 
By hypothesis, $M$ is a K\"ahler $G$-space, that is, 
$(\omega, g , \mathbf{J})$ is a $G$-invariant compatible triple. Recall that $\omega^n = n! \mu_g$, where
$\mu_g$ is the volume form associated to the
Riemannian metric $g$ (see, e.g., 
\cite[formula (4.20]{Ballmann2006}) and hence 
$\operatorname{L}^2_{\mu_g} = 
\operatorname{L}^2_{\omega^n}$. Take $\lambda = \mu_g$ and
hence $\delta_\lambda = \delta$ is the usual codifferential associated to the Riemannian metric $g$. 

Now repeat the proof of Theorem \ref{thm_gen}. In Step 1 
the crucial fact was that if a cohomology class has a
harmonic $\operatorname{L}^2$ representative, then it is
unique. In the hypotheses of the corollary, this is implied
by the weak $\operatorname{L}^2$-Hodge decomposition which holds for all complete non-compact
Riemannian manifolds (see \cite{Li2010} for the general 
$\operatorname{L}^p$ formulation; for the strong 
$\operatorname{L}^p$ version see \cite{Li2009}) and the
fact that every infinitesimal generator is a Killing vector field. Step 2 is unchanged.
Step 3 follows again from the weak $\operatorname{L}^2$-Hodge decomposition. Indeed, by hypothesis, the smooth 
closed one-form $\mathbf{i}_{ \xi_M} \omega \in 
\operatorname{L}^2_{\mu_g}$ 
for any $\xi\in \mathfrak{g}$ and hence it decomposes 
$\operatorname{L}^2_{\mu_g}$-orthogonally
as $\mathbf{i}_{\xi_M} \omega = \mathbf{d}f^\xi + 
\chi ^\xi$, where $f ^\xi\in \operatorname{C} ^{\infty}(M)$
and $\chi^\xi \in \Omega^1(M)$ is harmonic, $\| \mathbf{d} f ^\xi\|_{\operatorname{L}^2_{\mu_g}}<\infty$, $\| \chi ^\xi\|
_{\operatorname{L}^2_{\mu_g}}<\infty$. As before, 
$\mathbf{i}_{\xi_M} \omega = \mathbf{J} \xi_M^\flat$ and for
any harmonic $\alpha \in \Omega^1(M)$, $\| \alpha\|_{\operatorname{L}^2_{\mu_g}}<\infty$, we have $\left\langle \mathbf{i}_{\xi_M} \omega, \alpha 
\right\rangle_{\operatorname{L}^2_{\mu_g}} = 
-\left\langle \xi_M^\flat, \mathbf{J}\alpha 
\right\rangle_{\op{L}^2_{\mu_g}}$. Since $M$ is K\"ahler,
$\mathbf{J}\alpha$ is also harmonic (see, e.g., \cite[Cor 4.11, Ch. 5]{Wells2008}). Thus, by Step 2, 
$\left\langle \xi_M^\flat, \mathbf{J}\alpha 
\right\rangle_{\op{L}^2_\lambda} = 0$, which shows that 
$\chi^\xi = 0 $. Step 4 is unchanged.

\section{Examples}

The previous results apply in the following  examples.

\subsection{K\"ahler quotients} \label{11}

A large class of examples can be obtained using a
construction that will be presented below. In all that follows we assume that the manifolds are second countable. 
We begin with  a few preliminary remarks.

Let $\Gamma$ be a group that acts properly discontinuously
on a manifold $M $, that is, each $x \in M $ has 
a neighborhood $U $ such that $(\gamma \cdot U) \cap U = 
\varnothing$ for all $\gamma\neq e$. In particular, the 
$\Gamma$-action is free. Then the orbit space $M/\Gamma$ is a smooth manifold and the projection $p: M \rightarrow M/\Gamma $ is both
a local diffeomorphism and a covering map. Assume that $\mu \in \Omega^n(M)$, $n = \dim M$, is a $\Gamma$-invariant volume 
form on $M $. Then $M/\Gamma$ has a unique volume form $\nu
\in \Omega^n(M/\Gamma)$ such that $p^\ast \nu = \mu$. It is worth
noting that the measure $m_{M/\Gamma}$ on $M/\Gamma$ associated to 
the volume form $\nu$ does not coincide with the push 
forward of the measure $m_M$ on $M $ associated to the 
volume form $\mu$. Recall that $m_M$ is defined by the 
requirement that 
$
\int_M \psi \operatorname{d}\!m_M = \int_M \psi \mu
$
for any $\psi \in \operatorname{C}^{\infty}(M)$ with compact support, where the integral on the left is relative
to the measure $m_M $ and the integral on the right is
relative to the volume form $\mu$. A similar definition
holds for $m_{M/\Gamma}$.

Let $F$ be a fundamental domain of the $\Gamma$-action 
on $M$, that is, $F \subset M $ is a set such that 
each $\Gamma$-orbit
intersects it in a single point. In particular, the restriction of the projection $p$ to $F$ gives a  
bijection between $F $ and $M/\Gamma$. Assume: 
\begin{itemize}
\item[(i)] the interior $\operatorname{int}(F) \neq\varnothing$
\item[(ii)] the point set theoretical boundary $\overline{F} \setminus \operatorname{int}(F)$ has zero $m_M $-measure.
\end{itemize}
\textbf{Claim 1:} For any 
$m_{M/\Gamma}$-integrable $f \in 
\operatorname{C}^{\infty}(M/\Gamma)$ 
we have
\begin{equation}
\label{integral_quotient}
\int_{M/\Gamma} f \operatorname{d}\!m_{M/\Gamma} 
= \int_{\operatorname{int}(F)}
(f \circ p) \operatorname{d}\!m_M 
= \int_{\overline{F}}(f \circ p) \operatorname{d}\!m_M.
\end{equation}
Note that these integrals are \textit{not} identically equal to zero 
by hypothesis (i).

To see this we begin by showing that the 
$m_{M/\Gamma}$-measure
of $M/\Gamma \setminus p( \operatorname{int}(F)) $ is zero. Indeed, since $p(\overline{F}) = M/\Gamma$, because $F $ is a fundamental domain, and $p\left(\overline{F}\right)\setminus 
p\left(\operatorname{int}(F)\right) \subseteq p \left(
\overline{F} \setminus \operatorname{int}(F) \right)$ we get
$
m_{M/\Gamma} \big(M/\Gamma \setminus p(\operatorname{int}(F)\big)
= m_{M/\Gamma} \big(p(\overline{F}) \setminus p(\operatorname{int}(F)\big) \leq
m_{M/\Gamma} \big(p( \overline{F} \setminus \operatorname{int}(F)
\big).
$
However, $m_M(\overline{F} \setminus \operatorname{int}(F))
=0$ by hypothesis (ii) and since the smooth map $p$ maps
measure zero sets to measure zero sets ($M$ is second
countable), it follows that $m_{M/\Gamma} \big(p( \overline{F} \setminus \operatorname{int}(F)\big)=0$.

Formula \eqref{integral_quotient} follows from a change of
variables and the remark above. Indeed, for any $m_{M/\Gamma}$-integrable $f \in \operatorname{C} ^{\infty}(M/\Gamma) $ we have
\begin{align*}
\int_{M/\Gamma} f \operatorname{d}\!m_{M/\Gamma} 
&= \int_{p(\operatorname{int}(F))}f \operatorname{d}\!m_M
= \int_{p(\operatorname{int}(F))}f\nu
= \int_{\operatorname{int}(F)}p ^\ast(f \nu)\\
& = \int_{\operatorname{int}(F)}(p ^\ast f)(p ^\ast \nu) \
= \int_{\operatorname{int}(F)}(f\circ p)\mu
=\int_{\operatorname{int}(F)}(f\circ p)\operatorname{d}\!m_M.
\end{align*}
This concludes the proof of Claim 1.

Assume now that a group $\Gamma$ acts on 
two volume manifolds $M_1$ and $M_2$ and that the action 
on $M_2$ is properly discontinuous. The the diagonal 
$\Gamma$-action on $M_1 \times M_2 $ is also properly 
discontinuous. Let $F \subseteq M_2$ be the fundamental 
domain of the $\Gamma$-action on $M_2 $ and assume (i) 
and (ii). The fundamental domain of the diagonal action 
is easily verified to equal $M_1 \times F$. Denote the 
measures on $M_1$, $M_2$, and $(M_1 \times M_2)/\Gamma$ by 
$m_1$, $m_2$, and $q $, respectively. 

\noindent \textbf{Claim 2:} Assume that 
$m_2(\overline{F}) < \infty$. Then, for any $q$-integrable
$f \in \operatorname{C} ^{\infty}((M_1 \times M_2)/\Gamma)$ 
such that $f \circ p $ does not depend on $M_2$, we
have 
\begin{equation}
\label{important_inequality}
\int_{(M_1 \times M_2)/\Gamma} f \operatorname{d}\!q
= m_2(\overline{F}) \int_{M_1} (f \circ p) 
\operatorname{d}\!m_1.
\end{equation}
Indeed, denoting by $\chi_{\overline{F}} $ the characteristic
function of $\overline{F}$, using the Fubini theorem we get
\begin{align*}
\int_{(M_1 \times M_2)/\Gamma} f \operatorname{d}\!q
&\stackrel{\eqref{integral_quotient}}=
\int_{\overline{M_1 \times F}}(f \circ p) \operatorname{d}
(m_1 \times m_2)
= \int_{M_1 \times M_2} \chi_{\overline{F}}(f \circ p)
\operatorname{d}(m_1 \times m_2)\\
& = \int_{M_1} \left[\int_{M_2} \chi_{\overline{F}}
(f \circ p) \operatorname{d}\!m_2 \right]
\operatorname{d}\!m_1
= \int_{M_1}(f \circ p)\left[\int_{M_2} \chi_{\overline{F}}
\operatorname{d}\!m_2 \right]
\operatorname{d}\!m_1\\
&= m_2(\overline{F})\int_{M_1} (f\circ p)
\operatorname{d}\!m_1.
\end{align*}
 
Now we construct a class of examples for Corollary 
\ref{nicecor}. Let a compact Lie group $G $ act by K\"ahler transformations
on a compact  K\"ahler manifold $M_1$.
 Suppose that the $G$-action on $M_1$ has fixed points. 
Let $M_2$ be a complete (possibly noncompact) 
K\"ahler manifold with finite
volume. Let $\rho \colon \pi_1(M_2)\times M_1 \rightarrow M_1$ be a K\"ahler action which commutes with 
the $G$-action on $M_1$.  
Let $\pi:\pi_1(M_2) \times M_2 \rightarrow M_2$ be the natural action by covering translations of 
$\pi_1(M_2)$ which is K\"ahler. Then the diagonal action
$\pi \times \rho \colon \pi_1(M_2) : M_1 \times M_2 \rightarrow M_1 \times M_2$ is also K\"ahler and commutes
with the $G$-action.
Let the twisted product $\mathcal{M}
:=(M_1 \times \widetilde{M}_2)/(\pi \times \rho)$
be equipped with the K\"ahler structure 
inherited from the product symplectic form
on $M_1 \times \widetilde{M}_2$. Denote by $\omega$
the symplectic form on $\mathcal{M}$. Let $G$ act on
$\mathcal{M}$ by means of the $G$\--action on $M_1$,
leaving $M_2$ fixed. Then the $G$-action
on the complete K\"ahler manifold $\mathcal{M}$ has fixed points. Since $M_1$ is compact and $m_2(\overline{F}) < \infty$, we have 
\begin{eqnarray}
\| \mathbf{i}_{ \xi_M} \omega\|_{ \operatorname{L}^2} =
\int_{(M_1 \times \widetilde{M}_2)/(\pi \times \rho)}  \| {\mathbf i}_{\xi_{ \mathcal{M}}}\omega \|^2 \operatorname{d}\!q 
&\stackrel{\eqref{important_inequality}}
=& m_2( \overline{F}) \int_{M_1} \left(
\| {\mathbf i}_{\xi_{ \mathcal{M}}}\omega \|^2 \circ p \right) \operatorname{d}\!m_1 <\infty.
\end{eqnarray}
Hence the assumption of Corollary \ref{nicecor} are satisfied
and hence the $G $-action on $\mathcal{M}$ is Hamiltonian.

Another way this example could have been treated is by
applying Frankel's original theorem for compact K\"ahler
manifolds to $M_1 $ and then using reduction.

Let us spell out a concrete example of this situation.
Let $M_1:=S^2$
with the standard Fubini-Study form and the standard
Hamiltonian $S^1$\--action with $2^n$ fixed points.
Let $M_2:=\Sigma_{\infty}$ be a complete K\"ahler symplectic $2$\--manifold
of infinite genus and finite symplectic volume.
Let $\pi_1(\Sigma_{\infty})$
act on $\widetilde{\Sigma}_{\infty}$ by covering translations
and let $\pi \times \rho$ act on $S^2 \times \widetilde{\Sigma}_{\infty}$ 
by the diagonal action, i.e.
$\pi \times \rho \colon \pi_1(\Sigma_{\infty})
\to  \op{SO}(3) \times \op{Sympl}(\widetilde{\Sigma}_{\infty})$, $g \mapsto (\pi(g),\, \rho(g))$,
where $\rho \colon \pi_1(\Sigma) \to \op{SO}(3)$
is any representation that commutes with the
Hamiltonian $S^1$\--action on $S^2$ given by rotations
about the vertical axis, in other words, that $\rho$
factors through $\op{SO}(2)$. This action is by K\"ahler automorphisms,
so the quotient $(S^2 \times \widetilde{\Sigma}_{\infty})/(\pi\times \rho)$ is K\"ahler.

\subsection{Manifolds constructed by symplectic sum}

One can also construct examples that satisfy the assumption of Corollary \ref{nicecor}
using Gompf's symplectic sum \cite{Gompf1995}.  These examples can be viewed as a sub-collection
of those given in Section \ref{11}, so we do not provide details. Nonetheless, this
construction is different, so it is worth presenting and outline. Let $k >1$ and $\omega_{\textup{FS}}$ be the Fubini\--Study symplectic form
on $\mathbb{C}P^{1}$ and let $\omega_{\mathbb{T}^2}$
be the standard area form on the $2$\--torus. 
Let the finite additive group $\mathbb{Z}/q \mathbb{Z}$, $q \in \mathbb{N}$, act on $(\mathbb{C}P^1)^{n-1}$
by rotating $360/q$ degrees on one or more copies of $\mathbb{C}P^1$
inside of $(\mathbb{C}P^1)^{n-1}$, and by rotations of $360/q$ degrees on the first (or both)
standard sub-circles of $\mathbb{T}^2$. This gives rise to a diagonal symplectic action of $\mathbb{Z}/q \mathbb{Z}$
on the product manifold $(\mathbb{C}P^1)^{n-1} \times \mathbb{T}^2$ equipped with the symplectic form $(n-1) \omega_{\textup{FS}} \oplus \frac{1}{m^k}\, \omega_{\mathbb{T}^2}$.
The quotient $(\mathbb{C}P^1)^{n-1} \times_{\mathbb{Z}/q \mathbb{Z}} \mathbb{T}^2$ is a smooth manifold endowed with
the quotient symplectic form denoted by $(n-1) \omega_{\textup{FS}} \oplus_q \frac{1}{m^k}\,\omega_{\mathbb{T}^2}$.

Consider any Hamiltonian  $S^1$\--action on
$(\mathbb{C}P^1)^{n-1}$, for example the action
that acts in the usual way on the first component
of the product, and acts trivially on the other
components. Assume that $S ^1$ acts trivially on the second factor $\mathbb{T}^2$. This gives rise to a $S^1$\--action on 
$(\mathbb{C}P^1)^{n-1} \times \mathbb{T}^2$.   Let $x_m$ ($m \in \mathbb{N}$) be arbitrary distinct points in
$\mathbb{T}^2$. Let $\pi(Y_m)$ be the
$S^1$\--invariant codimension\--two symplectic submanifold of 
 $(\mathbb{C}P^1)^{n-1} \times_{\mathbb{Z}/q \mathbb{Z}} \mathbb{T}^2$ given by projecting
$(\mathbb{C}P^1)^{n-1} \times \{x_m\}$ under the canonical projection
map $\pi:(\mathbb{C}P^1)^{n-1} \times \mathbb{T}^2 
\to (\mathbb{C}P^1)^{n-1} \times_{\mathbb{Z}/q \mathbb{Z}} \mathbb{T}^2$. 
The $S^1$\--action on $(\mathbb{C}P^1)^{n-1} \times \mathbb{T}^2$ gives 
rise to a $S^1$\--action on 
$(\mathbb{C}P^1)^{n-1} \times_{\mathbb{Z}/q \mathbb{Z}} \mathbb{T}^2$, with
``large'' fixed point sets, and with even larger
fixed point sets on the (infinite) connected symplectic sum
\begin{eqnarray} \label{exs}
\mathcal{M}_k:=\#_{\pi(Y_m),m \in \mathbb{N}} 
\Big((\mathbb{C}P^1)^{n-1} \times_{\mathbb{Z}/q \mathbb{Z}} \mathbb{T}^2, \,\,\,
(n-1) \omega_{\textup{FS}} \oplus_q \frac{1}{m^k}\, \omega_{\mathbb{T}^2}   \Big).
\end{eqnarray}
Let $\omega$ be the symplectic form
of $\mathcal{M}_k$. The space $\mathcal{M}_k$ is symplectomorphic to $(\mathbb{C}P^1)^{n-1} \times \Sigma_{\infty}$,
where $\Sigma_{\infty}$ is an infinite genus surface and has an area
form that ``decays'' at infinity. 
Thus the manifold $\mathcal{M}_k$ is non\--compact and K\"ahler.
By choosing an appropriate complete metric on $\Sigma_{\infty}$, $\mathcal{M}_k$ is complete. So Corollary 
\ref{nicecor} applies.

\subsection{Hodge decomposition with decay hypotheses}
For non-compact manifolds the Hodge decomposition does
not hold, in general. The first positive result
is Kodaira's decomposition on a
noncompact Riemannian manifold (see \cite{Kodaira1949},
\cite[p. 165]{deRham1973})
\[
\operatorname{L}^2(\wedge^q M, g) = 
\overline{\mathbf{d}\operatorname{C}^{\infty}_0
(\wedge^{q-1}M)}
\oplus \overline{\delta \operatorname{C}^{\infty}_0
(\wedge^{q+1}M)}
\oplus \mathcal{H}(\wedge^q M, g),
\]
where $\delta$ is the $\operatorname{L}^2$ formal adjoint
of the exterior derivative operator $\mathbf{d}$, the closures are taken in $\operatorname{L}^2 $, 
the subscript 0 denotes compact support, $\mathcal{H}(\wedge^q M, g): = \{\alpha \in \operatorname{L}^2
(\wedge^q M, g) \mid \mathbf{d}\alpha = 0, \delta \alpha = 0\}$, $\operatorname{L}^2 $ is taken relative to the 
measure associated to the Riemannian metric $g $, and the
sum is $\operatorname{L}^2 $-orthogonal. This kind of
Hodge decomposition is not enough for our purposes since
the best one can hope for, as the proof of Theorem 
\ref{thm_gen} shows, is the 
existence of smooth functions $f^\xi$ with compact support.
It is well known (see, e.g., \cite{Lockhart1987}) that, in general, the Laplacian is not a Fredholm operator,  that
the kernel of the Laplacian does not consists of closed 
and coclosed forms, and that it does not give a unique
or complete representation of the de Rham cohomology
groups.

Improvements on this theorem are possible only if one
changes the measure and the formal adjoint of $\mathbf{d}$
or puts restrictions on the manifold and the metric.
In both cases there are Hodge type decomposition theorems
that fit the hypotheses of Theorem \ref{thm_gen}. We have already seen such a situation in the Corollary \ref{nicecor}
when we used \cite{Li2009}. We will
present further corollaries of Theorem \ref{thm_gen} in both
situations in this subsection and the next.

\paragraph{Ahmed and Stroock \cite{AhSt2000} conditions.} Minimal requirements on the geometry of $M$ that give
a Hodge decomposition theorem and link harmonic forms
to de Rham cohomology are given by Ahmed and Stroock in
\cite[\S6]{AhSt2000}. The assumptions on the 
manifold $(M, g)$ are: 

\noindent \textbf{(AM)} \textit{$M$ is connected, complete, its Ricci curvature is
bounded below by $- \kappa_{\operatorname{Ric}}\leq 0$,
and the Riemann curvature operator is bounded above, i.e.,
$\left\langle \! \left\langle R \alpha, \alpha \right\rangle \! \right\rangle \leq \kappa \| \alpha\|_{ \operatorname{L}^2}$ for all $\alpha \in \Omega^2(M) $, where $\kappa\geq 0$.}

In addition, to get a Hodge decomposition, the Riemannian
measure needs to be changed by a factor 
$\operatorname{e}^{-U}$, where $U $ satisfies certain
conditions.

\noindent \textbf{(AU)} \textit{Let $U : M \rightarrow [0, \infty)$ be a smooth function with the following properties:
\begin{itemize}
\item $U $ has compact level sets
\item There exists $C<\infty$ and $\theta\in (0,1) $ such
that $\Delta U \leq C(1+U) $ and $\|\nabla U\|^2 \leq C
\operatorname{e}^{ \theta U}$.
\item There exists $\varepsilon > 0$ such that $\varepsilon
U^{1+ \varepsilon} \leq 1+\| \nabla U\|^2$.
\item There exists a $B<\infty$ such that $\left\langle \! \left\langle v_x, 
(\nabla^2 U)(v_x) \right\rangle \! \right\rangle \geq 
-B\|v_x\|^2$ for all $x \in M $ and $v_x \in 
\operatorname{T}_xM$.
\end{itemize} }

Let $\delta^U: \Omega^{k+1}(M) \rightarrow \Omega^k(M)$ be
defined by $\delta^U \alpha: = \operatorname{e}^U \delta (
\operatorname{e}^{-U} \alpha)$, $\Delta^U = \delta^U \mathbf{d} + \mathbf{d}\delta^U $, $\mathcal{H}^U : = 
\ker \Delta^U $, and $\operatorname{d}\!\lambda_U: = \operatorname{e}^{-U}
\operatorname{d}\!\lambda(g)$, where $\operatorname{d}\!\lambda(g)$ is the Riemannian measure
associated to $g$. Under these conditions the Radon-Nikodym derivative $\operatorname{e}^{-U}$ of
the Riemannian measure relative to $\operatorname{d}\!\lambda_U$ is bounded
on $M$. Relative to the measure $\operatorname{d}\!\lambda_U $ 
one associates the $\operatorname{L}^2_U$-spaces on forms.

We shall need the following results from 
\cite{AhSt2000}. Under the hypotheses \textbf{(AM)} and
\textbf{(AU)} we have the following results:
\begin{enumerate}
\item[(1)] \textit{Any closed $\operatorname{L}_U^2$-form $\alpha \in 
\Omega^1(M)$ has a unique $\operatorname{L}_U^2$-orthogonal decomposition $\alpha = \mathbf{d}f + \chi $, where $f $
is a $\operatorname{L}_U^2$ smooth function, $\chi \in \Omega^1(M) $, $\Delta^U \chi = 0$ and $\chi $ is of class
$\operatorname{L}_U^2$} (see \cite[Theorem 5.1]{AhSt2000}.
\item[(2)] \textit{Each cohomology class $[\alpha]\in  
\operatorname{H} ^1(M, \mathbb{R})$ has a unique smooth
$\operatorname{L}_U^2$ representative in $\ker \Delta^U$}
(see \cite[Theorem 6.4]{AhSt2000}).
\end{enumerate}

Choose
a  $G $-invariant Riemannian metric and a $G $-equivariant
complex structure $J $ such that $( \omega, g, \mathbf{J}) $
is compatible.

\begin{cor}\label{Stroock_conditions}
Assume that the compact Lie group acts on the connected symplectic
manifold $(M, \omega)$ preserving the symplectic structure.
Let $(\omega, g , \mathbf{J})$ be a compatible $G $-invariant
triple and assume hypothesis {\rm \textbf{(AM)}} holds. 
Let $U $ be any $G $-invariant function satisfying {\rm \textbf{(AU)}} and suppose that $\mathbf{J}\mathcal{H}^U \subset \mathcal{H}^U $. If the $G$-action has fixed points, it is Hamiltonian.
\end{cor}

\begin{proof}
Let $\dim M = 2n$. 
We shall verify the conditions in Theorem \ref{thm_gen}. 
We show that $\delta^U $ is the formal adjoint of 
$\mathbf{d}$ relative to the $\operatorname{L}^2_U$ inner
product. Indeed, if $\alpha \in \Omega^k(M)$ and $\beta
\in \Omega^{k+1}(M)$ are such that 
$\| \alpha\|_{\operatorname{L}^2_U}< \infty$, 
$\| \beta\|_{\operatorname{L}^2_U}< \infty$, $k \geq 0$,
we have
\begin{align*}
\left\langle \alpha, \delta^U \beta 
\right\rangle_{\operatorname{L}^2_U}
&=\int_M \left\langle \! \left\langle \alpha,  \operatorname{e}^U
\delta(\operatorname{e}^{-U} \beta) \right\rangle \! \right\rangle
\operatorname{e}^{-U} \operatorname{d}\! \lambda(g)
=\int_M \left\langle \! \left\langle \alpha,  
\delta(\operatorname{e}^{-U} \beta) \right\rangle \! \right\rangle
\operatorname{d}\! \lambda(g)\\
&= \int_M \alpha\wedge *\delta(\operatorname{e}^{-U}\beta)
= \int_M \mathbf{d} \alpha \wedge * \operatorname{e}^{-U}
\beta - \int_M \mathbf{d}(\alpha \wedge * 
\operatorname{e}^{-U} \beta).
\end{align*}
The second terms vanishes by Stokes' Theorem since $\partial
M = \varnothing$ and hence we get
\begin{equation}\label{identity_inner_product}
\left\langle \alpha, \delta^U \beta 
\right\rangle_{\operatorname{L}^2_U}
=\int_M \left\langle \! \left\langle \mathbf{d} \alpha, 
\operatorname{e}^{-U} \beta \right\rangle \! \right\rangle 
\operatorname{d}\! \lambda(g)
= \int_M \left\langle \! \left\langle \mathbf{d} \alpha, 
\beta \right\rangle \! \right\rangle \operatorname{e}^{-U} 
\operatorname{d}\! \lambda(g)
= \left\langle \mathbf{d}\alpha, \beta 
\right\rangle_{\operatorname{L}^2_U}
\end{equation}
as required.

Next, we show that $\ker \Delta^U = \ker \mathbf{d} \cap
\ker \delta^U$, which shows that a form is harmonic as
defined in Theorem \ref{thm_gen} (i) if and only if 
$\Delta^U$ vanishes on it. We begin with the identity
\[
\int_M \mathbf{d} \alpha \wedge * \operatorname{e}^{-U} \beta
= \int_M \alpha\wedge * \delta(\operatorname{e}^{-U} \beta)
\]
which is equivalent to \eqref{identity_inner_product}. If $\Delta^U \alpha = 0$, we have
\begin{align*}
0&= \int_M \left(\Delta^U \alpha \wedge * \alpha\right) \operatorname{e}^{-U} \operatorname{d}\! \lambda(g)
= \int_M \left(\mathbf{d} \delta^U\alpha \wedge * \alpha\right) \operatorname{e}^{-U} \operatorname{d}\! \lambda(g) + 
\int_M \left(\delta^U \mathbf{d}\alpha \wedge *\alpha\right) \operatorname{e}^{-U} \operatorname{d}\! \lambda(g)\\
& = \int_M \left\langle \! \left\langle \mathbf{d} \delta^U\alpha, \alpha \right\rangle \! \right\rangle\operatorname{e}^{-U}\operatorname{d}\! \lambda(g)
+ \int_M \left\langle \! \left\langle \delta^U\mathbf{d} \alpha, \alpha \right\rangle \! \right\rangle\operatorname{e}^{-U} \operatorname{d}\! \lambda(g)s
= \left\langle \mathbf{d}\delta^U\alpha, \alpha 
\right\rangle_{\operatorname{L}^2_U}
+ \left\langle \delta^U\mathbf{d}\alpha, \alpha 
\right\rangle_{\operatorname{L}^2_U}\\
& \stackrel{\eqref{identity_inner_product}}=
\left\langle \delta^U\alpha, \delta^U\alpha 
\right\rangle_{\operatorname{L}^2_U}
+ \left\langle \mathbf{d}\alpha, \mathbf{d}\alpha 
\right\rangle_{\operatorname{L}^2_U}.
\end{align*}
Therefore $\|\mathbf{d}\alpha\|_{\operatorname{L}^2_U}=0$
and $\|\delta^U\alpha\|_{\operatorname{L}^2_U}=0$ which
implies that $\mathbf{d}\alpha = 0 $ and $\delta^U \alpha
=0 $ since $\alpha$ is smooth. This proves that 
$\ker\Delta^U \subset \ker \mathbf{d} \cap \ker \delta^U $. 
The converse inclusion is obvious.

This shows that the hypotheses of Theorem \ref{thm_gen}
are satisfied and hence, provided that the $G $-action has
fixed points, we can conclude that it is Hamiltonian.
\end{proof}

\paragraph{Gong and Wang \cite{GoWa2004} conditions.} The
Hodge decomposition holds also under different conditions
involving the decay of the measure, as given in
\cite[Theorem 1.4]{GoWa2004}. Let the compact Lie group $G$ act on the
non-compact symplectic manifold $(M, \omega)$ by symplectic 
diffeomorphisms. Let $(\omega, g, \mathbf{J})$ be a 
$G$-invariant compatible triple. Denote by $\mathcal{R}$ the 
curvature term in the Weitzenb\"ock formula on one-forms and 
assume that $e^Vd\lambda(g)$ is a finite measure, where
$d\lambda(g) $ is the Riemannian volume associated to $g$.
Suppose that 
\begin{itemize}
\item $V $ is $G $-invariant
\item $\mathcal{R} - \operatorname{Hess}(V)$ is bounded 
below 
\item there exists a positive $G $-invariant function 
$U \in \operatorname{C}^2(M)$ such that $U+V$ is bounded
\item the sets $\{U \leq C\}$ are compact for all $C>0$
\item $\|\nabla U\| \rightarrow \infty$ as $U \rightarrow \infty $
\item $\operatorname{limsup}_{U \rightarrow \infty} 
\left(\Delta U/\|\nabla U\|^2 \right)< 1$
\end{itemize}
Then the Hodge decomposition holds, as shown in
\cite[Theorem 1.4]{GoWa2004}. Proceeding as before, we get
the following result.

\begin{cor}\label{GoWa_conditions}
Assume that the compact Lie group acts on the connected symplectic
manifold $(M, \omega)$ preserving the symplectic structure.
Let $(\omega, g , \mathbf{J})$ be a compatible $G $-invariant
triple and assume the above hypotheses. If $\mathbf{J}$ preserves the space of harmonic one-forms and the $G$-action has fixed points, it is Hamiltonian.
\end{cor}

\subsection{Manifolds with two ends}
Another class of examples is obtained by putting 
conditions on $M $. This class of 
examples appears in the work 
of Lockhart \cite[Example 0.16]{Lockhart1987}. We 
assume that $M $ has finitely
many ends, which means that $M $ contains a compact
submanifold $M_0$ whose smooth boundary $\partial M_0$
has finitely many components such that $M\setminus M_0 
= \partial M_0 \times (0, +\infty)$.
There is a natural additive monoid action of $[0, +\infty)$
on $\partial M_0 \times (0, +\infty)$ by translations
on the right factor. We say that a metric on $M$
is \emph{translation invariant} if on  
$\partial M_0 \times (0, +\infty)$ is invariant under
this action. Now suppose that $h_{\infty}$ is a translation
invariant metric, and let $\operatorname{D}_{\infty}$ denote 
the covariant derivative of the Levi-Civita connection associated to $h_\infty$. A metric $h$ is \emph{asymptotic
to $h_{\infty}$} if for each $k \in \mathbb{N}$
we have that
$$
\lim_{z \to \infty} \sup_{\omega \in \partial M_0}
\|\operatorname{D}^k_{\infty}h(\omega,z)-
\operatorname{D}^k_{\infty}h_{\infty}(\omega)
\|_{h_{\infty}}=0.
$$
The metric  $h$ is \emph{asymptotically translation 
invariant} if $h$ is asymptotic to a translation
invariant metric. Now let $h$ be an asymptotically
invariant metric. Let $g=\op{e}^{-2\rho}h$ with $\rho \in 
\op{C}^{\infty}(M)$. We say that the metric $g$ is
\emph{admissible} if there is a smooth, 
$(0, +\infty)$-invariant $1$-form $\theta$ on
$M_{\infty}$ with the property that
$$
\lim_{z \to \infty} \sup_{\omega \in \partial M_0}
\|\operatorname{D}^{k+1}_{(h)}\rho-
\operatorname{D}^k_{(h)}\theta \|_h=0,
$$
where $\operatorname{D}_{(h)}$ denotes the covariant derivative of the Levi-Civita connection associated to $h$.

\begin{cor}
Let $(M, \omega)$ be a connected symplectic manifold of
dimension at least four which has two
ends. Let $(\omega, g, \mathbf{J})$ be a $G $-invariant compatible triple.
Assume that $g=\op{e}^{2\rho}h$ is admissible, where 
$h$ is an asymptotically translation invariant metric,
$\rho(w,z)$ is decreasing,  and $\rho(\omega,z)
<-[(1+\epsilon)/2]\ln z$ on $\partial M_0 \times
[1,\infty)$ for some $\epsilon>0$. If the
$G$-action has fixed points, then it is Hamiltonian.
\end{cor}

The proof of the corollary follows from 
Theorem \ref{thm_gen} because there is 
a Hodge decomposition theorem for these manifolds
and each cohomology class in $\op{H}^1(M,\mathbb{R})$ 
has a unique harmonic representative, see
\cite[Formula (0.16.1)]{Lockhart1987}.

{\footnotesize

\bibliographystyle{new}

}

\smallskip\noindent
\'Alvaro Pelayo\\
University of California, Berkeley\\
Mathematics Department\\
Berkeley, CA 94720-3840, USA\\
{\em E\--mail}: \texttt{apelayo@math.berkeley.edu} 
\bigskip

\smallskip\noindent
Tudor S. Ratiu\\
Section de Math\'ematiques and Bernoulli Center\\
Station 8\\
Ecole Polytechnique F\'ed\'erale
de Lausanne\\
 CH-1015 Lausanne, Switzerland\\
{\em E\--mail}: \texttt{tudor.ratiu@epfl.ch}
\end{document}